\documentclass[11pt,a4paper,reqno]{amsart}

\usepackage[margin=3cm]{geometry}
\usepackage{amsmath,amsfonts,amssymb,amsthm}
\usepackage{mathtools}
\usepackage{url}
\usepackage[dvipsnames]{xcolor}
\usepackage[colorlinks=true,linkcolor=Blue,citecolor=Green,unicode]{hyperref}
\usepackage{aliascnt}       
\usepackage[T1]{fontenc}
\usepackage[utf8]{inputenc}
\usepackage{enumitem}
\setlist[enumerate]{leftmargin=*}
\usepackage[capitalise]{cleveref}
\usepackage{algorithm}
\usepackage{caption}
\usepackage{nicematrix}

\usepackage[textwidth=2.5cm,textsize=small]{todonotes}

\newtheorem{theorem}{Theorem}[section]
\newaliascnt{lemma}{theorem}
\newaliascnt{corollary}{theorem}
\newaliascnt{definition}{theorem}
\newaliascnt{remark}{theorem}
\newaliascnt{proposition}{theorem}
\newaliascnt{conjecture}{theorem}
\newaliascnt{example}{theorem}
\newaliascnt{question}{theorem}
\newaliascnt{claim}{theorem}

\newtheorem{lemma}[lemma]{Lemma}
\aliascntresetthe{lemma}

\newtheorem*{lemma*}{Lemma}

\newtheorem{corollary}[corollary]{Corollary}
\aliascntresetthe{corollary}

\newtheorem*{corollary*}{Corollary}

\newtheorem{definition}[definition]{Definition}
\aliascntresetthe{definition}

\newtheorem*{definition*}{Definition}

\newtheorem{remark}[remark]{Remark}
\aliascntresetthe{remark}

\newtheorem*{remark*}{Remark}

\newtheorem{proposition}[proposition]{Proposition}
\aliascntresetthe{proposition}

\newtheorem*{proposition*}{Proposition}

\aliascntresetthe{conjecture}

\newtheorem*{conjecture*}{Conjecture}

\newtheorem{example}[example]{Example}
\aliascntresetthe{example}

\newtheorem*{example*}{Example}

\newtheorem*{problem*}{Problem}

\newtheorem*{mproblem*}{Main Problem}

\newtheorem{question}[question]{Question}
\aliascntresetthe{question}

\newtheorem*{question*}{Question}

\aliascntresetthe{claim}

\newtheorem*{claim*}{Claim}

\Crefname{figure}{Figure}{Figures}

\DeclareMathOperator{\inter}{int}

\DeclareMathOperator{\conv}{conv}

\DeclareMathOperator{\diag}{diag}

\def\R{\mathbb{R}}

\def\Z{\mathbb{Z}}

\def\N{\mathbb{N}}

\DeclareMathOperator{\vol}{vol}
\DeclareMathOperator{\Vol}{Vol}

\DeclareMathOperator{\lin}{lin}
\DeclareMathOperator{\GL}{GL}

\DeclarePairedDelimiter{\card}{\lvert}{\rvert}

\newcommand{\one}{\mathbf{1}}
\newcommand{\zero}{\mathbf{0}}

\newcommand{\crk}{\operatorname{cr}}
\newcommand{\crke}{\operatorname{cr_e}}
\newcommand{\crks}{\operatorname{cr_s}}
\newcommand{\crkh}{\operatorname{cr_h}}

\newcommand{\zangle}[1]{\left\langle #1 \right\rangle}
\newcommand{\Gam}[1]{\Gamma_{#1}}
\newcommand{\Gamt}[1]{\Gamma^\intercal_{#1}}
\newcommand{\qG}[1]{G_{#1}}
\newcommand{\qGt}[1]{G^\intercal_{#1}}

\numberwithin{equation}{section}

\begin{document}

\title{The cyclicity rank of empty lattice simplices}

\author{Lukas Abend}
\address{Institute of Mathematics\\
 University of Rostock\\
 Germany}
\email{lukas.abend97@gmail.com}

\author{Matthias Schymura}
\address{BTU Cottbus-Senftenberg\\
  Platz der Deutschen Einheit 1\\
  03046 Cottbus\\
  Germany}
\email{matthiasmschymura@gmail.com}

\thanks{We thank Frieder Ladisch and Gennadiy Averkov for useful discussions that helped to clarify some of our arguments. This paper grew out from the master thesis~\cite{abend2023masterthesis} of the first author. We moreover thank two anonymous reviewers for their very valuable suggestions and constructive feedback.}




\begin{abstract}
We are interested in algebraic properties of empty lattice simplices~$\Delta$, that is, $d$-dimensional lattice polytopes containing exactly $d+1$ points of the integer lattice~$\Z^d$.
The cyclicity rank of~$\Delta$ is the minimal number of cyclic subgroups that the quotient group of~$\Delta$ splits into.
It is known that up to dimension $d \leq 4$, every empty lattice $d$-simplex is cyclic, meaning that its cyclicity rank is at most~$1$.
We determine the maximal possible cyclicity rank of an empty lattice $d$-simplex for dimensions $d \leq 8$, and determine the asymptotics of this number up to a logarithmic term.
\end{abstract}

\maketitle

\section{Introduction}

A \emph{lattice polytope} is a polytope in~$\R^d$ that has all its vertices in the integer lattice~$\Z^d$.
Empty lattice simplices, i.e., lattice simplices that do not contain points of~$\Z^d$ additionally to their vertices, are the building blocks of lattice polytopes with respect to triangulations.
Their classification in small dimensions proved useful to derive properties of all lattice polytopes; see the introduction of~\cite{iglesiassantos2021thecomplete} for more information and a list of references.
Besides this fundamental motivation to understand the class of empty lattice simplices, they showed to be relevant in quite a number of different contexts:
In algebraic geometry they are in close connection to terminal quotient singularities arising in the minimal model program.
This program has been completed in dimension three by Mori~\cite{mori1988fliptheorem}, based on White's~\cite{white1964lattice} classification of empty lattice tetrahedra.
The recent advances around the flatness problem in the Geometry of Numbers partially draw from the theory of empty lattice simplices as well (cf.~\cite{codenottisantos2020hollowpolytopes}).
Moreover, the three-dimensional classification plays a prominent role for certain problems around continued fractions in number theory (see, for instance, \cite{karpenkovustinov2017geometry}).

In this paper, we focus on the algebraic properties of empty lattice simplices, in particular, on the structure of their associated quotient groups.
These finite abelian groups arise as quotients of the integer lattice~$\Z^d$ by the sublattice that is spanned by the edge directions of the lattice simplex at hand.
Barile et al.~\cite{barilebernardiborisovkantor2011onempty} showed that the quotient group of a four-dimensional empty lattice simplex is always cyclic (dimension $d = 2$ is immediate, and in dimension $d = 3$ this follows from White's classification).
They also provide an example which shows that this result does not extend to dimensions $d \geq 5$.
However, the cyclicity of the quotient groups in dimension four was instrumental for Iglesias-Vali\~{n}o \& Santos~\cite{iglesiassantos2021thecomplete} to complete the classification of empty lattice $4$-simplices; a program that was started in the late 1980s.

It is just natural to expect that any approach to classify relevant families of empty lattice $5$-simplices will benefit from further insight into the algebraic structure of their quotient groups.
With this in mind, we introduce and investigate the \emph{cyclicity rank} of an empty lattice $d$-simplex~$\Delta$, which is the minimal number of cyclic groups that the quotient group of~$\Delta$ factors into.
Our main interest is about the maximal cyclicity rank~$\crke(d)$ that can occur in a given dimension~$d$ (see \cref{sect:basics-on-quotient-groups} for precise definitions).

Our results on the function~$d \mapsto \crke(d)$ can be summarized as follows:
In \cref{thm:crk-2-exact} and \cref{cor:crk5}, we show that for $d \geq 4$, the inequalities
\[
d - \lfloor \log_2(d) \rfloor - 1 \leq \crke(d) \leq d - 3 \,,
\]
hold.
Furthermore, we obtain the exact numbers in \cref{sect:p=2,sect:p=3} for small dimensions:
\begin{table}[H]
\begin{tabular}{c|c|c|c|c|c|c|c|c|c}
$d$        & $1$ & $2$ & $3$ & $4$ & $5$ & $6$ & $7$ & $8$ & $9$        \\ \hline
$\crke(d)$ & $0$ & $0$ & $1$ & $1$ & $2$ & $3$ & $4$ & $5$ & $5$ or $6$
\end{tabular}
\end{table}
In particular, these results show that it is only a small-dimensional phenomenon that empty lattice simplices have (almost) cyclic quotient groups.

\section{Quotient Groups of Lattice Simplices and Problem Statement}
\label{sect:basics-on-quotient-groups}

In this section, we introduce the basic definitions needed for the rest of the paper and precisely state the main problem that we are interested in.
Let's start with some standard terminology:
Given a subset $S \subseteq \R^d$, we denote by $\conv(S)$ the convex hull of~$S$ and with $\inter(S)$ the interior of~$S$ with respect to the standard topology on~$\R^d$.
Furthermore, if~$S$ is Lebesgue measurable, then $\vol(S)$ denotes its \emph{volume} (i.e., its Lebesgue measure) and $\Vol(S) = d!\vol(S)$ denotes its \emph{normalized volume}.
The standard unit vectors in $\R^d$ are given by $e_1,\ldots,e_d$, and with $[r] := \{1,2,\ldots,r\}$ we abbreviate the set of the first~$r$ natural numbers.

For us, a \emph{lattice} $\Lambda \subseteq \R^d$ is a discrete subgroup of full rank, the most prominent example being the standard integer lattice~$\Z^d$.
Most of the lattices that we consider in the sequel are \emph{sublattices} $\Lambda \subseteq \Z^d$ of the integer lattice, and they are uniquely defined as $\Lambda = A \Z^d$, for some matrix $A \in \Z^{d \times d}$ that is invertible over the rationals.
Such a matrix~$A$ is called a \emph{basis} of~$\Lambda$.
Note, however, that lattice bases are not unique.
Given a lattice $\Lambda$ its \emph{polar lattice} is defined by
\[
\Lambda^\star = \left\{ x \in \R^d : x^\intercal z \in \Z, \,\forall z \in \Lambda \right\}\,.
\]
If $A$ is a basis of~$\Lambda$, then $A^{-\intercal}$ is a basis of~$\Lambda^\star$; and for lattices $\Lambda \subseteq \Gamma$ polarity reverses the inclusion, that is, $\Gamma^\star \subseteq \Lambda^\star$.
Finally, we say that a vector $v = (v_1,\ldots,v_d)^\intercal \in \Z^d$ is \emph{primitive}, if the line segment joining~$v$ and the origin~$\zero$ contains exactly two lattice points, which is equivalent to saying that $\gcd(v_1,\ldots,v_d) = 1$.
We refer to the book of Gruber~\cite{gruber2007convex} for a thorough introduction to lattices, convex polytopes, and their interaction.

\begin{definition}
A \emph{lattice simplex}~$\Delta$ in~$\R^d$ is the convex hull of $d+1$ affinely independent lattice points, that is, there are affinely independent $v_0,v_1,\ldots,v_d \in \Z^d$ such that $\Delta = \conv\{v_0,v_1,\ldots,v_d\}$.
Such a lattice simplex $\Delta \subseteq \R^d$ is called \emph{hollow}, if $\inter(\Delta) \cap \Z^d = \emptyset$, and it is called \emph{empty}, if $\Delta \cap \Z^d = \{v_0,v_1,\ldots,v_d\}$.
\end{definition}

We say that two lattice simplices $\Delta,\Delta' \subseteq \R^d$ are \emph{unimodularly equivalent}, denoted by $\Delta \simeq \Delta'$, if there exists a unimodular matrix $U \in \GL_d(\Z)$ and a translation vector $t \in \Z^d$ such that $\Delta' = U \Delta + t$.
Because $U \Z^d = \Z^d$, for every $U \in \GL_d(\Z)$, we get that the properties of a lattice simplex of being empty or hollow are invariant under unimodular equivalence.
Note that we did not fix the labelling of the vertices of the simplices, so that this notion of equivalence includes the possibility to permute them.
This is relevant for comparison with the notion of the Hermite normal form of square integer matrices (see~\cref{thm:HNF}).

For our purposes we can assume without loss of generality that the origin~$\zero$ is one of the vertices of any given lattice simplex $\Delta \subseteq \R^d$, so that we can associate the invertible integer matrix $A = (v_1,\ldots,v_d) \in \Z^{d \times d}$ to $\Delta = \conv\{\zero,v_1,\ldots,v_d\}$.
Writing
\[
S_d := \conv\{\zero,e_1,\ldots,e_d\} = \left\{ x \in \R^d : \sum_{i=1}^d x_i \leq 1 \textrm{ and } x_i \geq 0 \textrm{ for every } 1 \leq i \leq d \right\}
\]
for the \emph{standard simplex}, we get the representation $\Delta = A S_d$.

Now, given a lattice simplex $\Delta = A S_d$, it induces two sublattices
\[
\Gam{\Delta} := A \Z^d = \Z a_1 + \ldots + \Z a_d \subseteq \Z^d \quad\textrm{ and }\quad \Gamt{\Delta} := A^\intercal \Z^d = \Z r_1 + \ldots + \Z r_d \subseteq \Z^d\,,
\]
where $a_1,\ldots,a_d \in \Z^d$ and $r_1,\ldots,r_d \in \Z^d$ are the columns and rows of the matrix~$A$, respectively.
The first sublattice~$\Gam{\Delta}$ can be seen as being the lattice that is generated by the edge directions of~$\Delta$.
Conversely, given a sublattice $\Gamma \subseteq \Z^d$ and a basis $B$ of~$\Gamma$, one may associate two lattice simplices to~$\Gamma$, whose non-zero vertices correspond to either the columns or the rows of~$B$.

A disadvantage of~$\Gam{\Delta}$ is that one cannot read off~$\Delta$ from a given basis of~$\Gam{\Delta}$ up to unimodular equivalence.
For instance, both $\{e_1,e_2,2 e_3\}$ and $\{e_1,e_2,(1,1,2)^\intercal\}$ are a basis of the lattice $\Gamma = \Z^2 \times (2 \Z)$, but the lattice simplices $\Delta = \conv\{\zero,e_1,e_2,2 e_3\}$ and $\Delta' = \conv\{\zero,e_1,e_2,(1,1,2)^\intercal\}$ are not unimodularly equivalent, because~$\Delta$ has a non-primitive edge and~$\Delta'$ is empty.
For the sublattice $\Gamt{\Delta}$, however, we can infer~$\Delta$ up to a finite explicit list of unimodular equivalence classes.
This makes the lattice $\Gamt{\Delta}$ the more natural choice when we want to identify lattice simplices with sublattices.
More precisely,

\begin{proposition}
\label{prop:lattice-simplex-vs-sublattice}
Let $\Delta,\Delta' \subseteq \R^d$ be lattice simplices which both have the origin as a vertex.
Then, we have
\[
\Delta \simeq \Delta' \quad\textrm{ if and only if }\quad \Gamt{\Delta} = P \, U_k \, \Gamt{\Delta'} \,,
\]
for some permutation matrix $P \in \Z^{d \times d}$ and some $k \in \{0,1,\ldots,d\}$, where $U_k \in \Z^{d \times d}$ is the unimodular matrix whose $(i,k)$-entries equal $-1$, for all $1 \leq i \leq d$, whose $(i,i)$-entries equal $1$, for all $1 \leq i \leq d$ with $i \neq k$, and whose remaining entries equal~$0$; in particular,~$U_0$ is the identity matrix.
\end{proposition}

\begin{proof}
Write $\Delta = A S_d$ and $\Delta' = B S_d$, for suitable integer matrices $A,B \in \Z^{d \times d}$.
Then, $\Delta \simeq \Delta'$ means that there is a matrix $U \in \GL_d(\Z)$ and a vector $t \in \Z^d$ such that $\Delta' - t = U \Delta$.
Since both simplices $\Delta,\Delta'$ have a vertex at the origin, $t$ must be a vertex of~$\Delta'$, which means that there is an index $k \in \{0,1,\ldots,d\}$ such that $t = B e_k$, where we write $e_0 = \zero$.
Observe that $S_d - e_k = U_k^\intercal S_d$, so that the conditions above can be written as $B U_k^\intercal S_d = UA S_d$.
This is equivalent to the existence of a permutation matrix $Q$ such that $B U_k^\intercal Q = UA$.
This in turn is equivalent to
\[
\Gamt{\Delta} = A^\intercal \Z^d = (U^{-1} B U_k^\intercal Q)^\intercal \Z^d = Q^\intercal U_k B^\intercal U^{-\intercal} \Z^d = Q^\intercal U_k B^\intercal \Z^d = Q^\intercal U_k \Gamt{\Delta'} \,,
\]
as $U^{-\intercal}$ is again unimodular.
\end{proof}

As a particular case of the previous statement, choosing two bases $V,W$ of a sublattice $\Gamma \subseteq \Z^d$ gives unimodular equivalent lattice simplices $V^\intercal S_d$ and $W^\intercal S_d$.
This property motivates our choice to consider the ``row-lattice''~$\Gamt{\Delta}$ of~$\Delta$, rather than the ``column-lattice''~$\Gam{\Delta}$.
In toric geometry however, the latter choice is the more natural (see, e.g.~\cite{barilebernardiborisovkantor2011onempty} or~\cite[Ch.~2]{fulton1993introduction}).
We see below, that for our purposes it doesn't make a difference whether we consider the integral row-span or column-span, so our results apply to the respective questions and properties in toric geometry without restriction.

Our main object of interest associated to a lattice simplex is its quotient group.
In fact, this comes in two versions:

\begin{definition}
\label{def:quotient-groups}
For a lattice simplex $\Delta = A S_d$, we let
\[
\qG{\Delta} := \Z^d / \,\Gam{\Delta} = \Z^d / (A \Z^d) \quad\textrm{ and }\quad \qGt{\Delta} := \Z^d / \,\Gamt{\Delta} = \Z^d / (A^\intercal \Z^d)
\]
be its \emph{quotient groups}.
\end{definition}

\noindent The quotient groups of~$\Delta = A S_d$ are finite abelian groups of the same order (cf.~\cite[Ch.~21]{gruber2007convex})
\[
\card{\qG{\Delta}} = [\Z^d : \Gam{\Delta}] = \card{\det(A)} = \card{\det(A^\intercal)} = [\Z^d : \Gamt{\Delta}] = \card{\qGt{\Delta}} \,.
\]
For any invertible matrix $M \in \Z^{d \times d}$ the quotient group $\Z^d / (M \Z^d)$ is determined by its so-called \emph{Smith normal form}, which is the unique diagonal matrix $S = \diag(m_1,\ldots,m_d)$ such that there exist unimodular matrices $U,V \in \GL_d(\Z)$ with $M = U S V$ and the $m_1,\ldots,m_d$ are positive integers satisfying $m_d \mid m_{d-1} \mid \cdots \mid m_1$ (cf.~Cohen~\cite[Ch.~2]{cohen1993acourse}).
These numbers $m_1,\ldots,m_d$ are called the \emph{elementary divisors} of~$M$.
In particular, we have
\begin{align}
\Z^d / (M \Z^d) &= \Z^d / (U S V \Z^d) = \Z^d / (U S \Z^d) \cong (U^{-1} \Z^d) / (S \Z^d) = \Z^d / (S \Z^d) \,,\label{eqn:quotient-smith}
\end{align}
and therefore
\begin{align}
\Z^d / (M \Z^d) \cong \Z_{m_r} \times \Z_{m_{r-1}} \times \ldots \times \Z_{m_1} \,,\label{eqn:quotient-smith-product}
\end{align}
where $r$ is the largest index such that the elementary divisor $m_r > 1$.
The Smith normal form helps us to derive the important observation that for every lattice simplex, the associated quotient groups in \cref{def:quotient-groups} are isomorphic:

\begin{proposition}
\label{prop:A-vs-A-transpose-quotients}
Let $A \in \Z^{d \times d}$ be an invertible integer matrix.
Then, the sublattices $A \Z^d$ and $A^\intercal \Z^d$ have isomorphic quotient groups in~$\Z^d$, that is,
\[
\Z^d / (A \Z^d) \cong \Z^d / (A^\intercal \Z^d) \,.
\]
In particular, for every lattice simplex $\Delta$, we have $\qG{\Delta} \cong \qGt{\Delta}$.
\end{proposition}

\begin{proof}
In view of~\eqref{eqn:quotient-smith} it suffices to show that~$A$ and~$A^\intercal$ have identical Smith normal forms.
To this end, let $U,V \in \GL_d(\Z)$ be such that $A = U S V$ and $S$ is the Smith normal form of~$A$.
Then, we have $A^\intercal = (U S V)^\intercal = V^\intercal S^\intercal U^\intercal = V^\intercal S U^\intercal$, so that indeed~$S$ is the Smith normal form of~$A^\intercal$ as well.
\end{proof}

Finally, let us note that the isomorphism class of the quotient groups $\qG{\Delta}$ and $\qGt{\Delta}$ only depend on the isomorphism class of a lattice simplex~$\Delta = A S_d$, and not, for instance, on the choice of the origin vertex.

\begin{proposition}
\label{prop:isomporphism-class-quotients}
Let $\Delta,\Delta' \subseteq \R^d$ be lattice simplices which both have the origin as a vertex.
If $\Delta \simeq \Delta'$, then $\qG{\Delta} \cong \qG{\Delta'}$ and $\qGt{\Delta} \cong \qGt{\Delta'}$.
\end{proposition}

\begin{proof}
In view of \cref{prop:A-vs-A-transpose-quotients} it suffices to establish only one of the claimed isomorphisms.
By \cref{prop:lattice-simplex-vs-sublattice} there is a permutation matrix $P$ and an index $k \in \{0,1,\ldots,d\}$ such that $\Gamt{\Delta} = P \, U_k \, \Gamt{\Delta'}$.
Therefore,
\[
\qGt{\Delta} = \Z^d / \, \Gamt{\Delta} = \Z^d / \, (P \, U_k \, \Gamt{\Delta'}) \cong ((P \, U_k)^{-1} \Z^d) / \, \Gamt{\Delta'} = \Z^d / \, \Gamt{\Delta'} = \qGt{\Delta'} \,,
\]
because $P \, U_k$ is unimodular.
\end{proof}

%
%
%
%

\subsection*{The cyclicity rank of (empty) lattice simplices}

A lattice simplex $\Delta \subseteq \R^d$ is called \emph{cyclic}, if its quotient group $G_\Delta$ is a cyclic group.
More generally, for $\Delta = A S_d$, we write $\crk(G_\Delta) = r$, if exactly~$r$ elementary divisiors of~$A$ are bigger than one (see~\eqref{eqn:quotient-smith-product}), that is, the quotient group~$G_\Delta$ of~$\Delta$ splits into~$r$ factors of cyclic subgroups and no less.
\Cref{fig:examples} shows two lattice triangles, one is cyclic while the other is not.

\begin{figure}[h]
\hfill\includegraphics[scale=1.2,page=1]{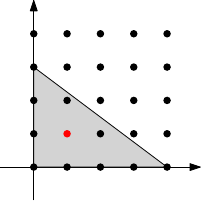}
\hfill\includegraphics[scale=1.2,page=2]{triangle-examples.pdf}\hfill\,
\caption{The triangle $\Delta = \conv\{\zero,\binom{4}{0},\binom{0}{3}\}$ on the left is cyclic with $G_\Delta \cong \Z_{12}$ generated by $\binom{1}{1}$; the red lattice point. The triangle $\Delta' = \conv\{\zero,\binom{3}{0},\binom{0}{3}\}$ however is not cyclic, and has $G_{\Delta'} \cong \Z_3 \times \Z_3$ with generators $\binom{1}{0}$ and~$\binom{0}{1}$; the two blue lattice points.}
\label{fig:examples}
\end{figure}

\begin{definition}[Cyclicity rank of a simplex]
\label{def:cyc-rank-simplex}
Given a lattice simplex $\Delta \subseteq \R^d$, we define its \emph{cyclicity rank} as $\crk(\Delta) := \crk(G_\Delta)$.
Moreover, we let
\begin{align*}
\crke(d) &:= \max\left\{\crk(\Delta) : \Delta \subseteq \R^d \text{ an empty lattice simplex}\right\} \,,\\
\crkh(d) &:= \max\left\{\crk(\Delta) : \Delta \subseteq \R^d \text{ a hollow lattice simplex}\right\} \,, \textrm{and}\\
\crks(d) &:= \max\left\{\crk(\Delta) : \Delta \subseteq \R^d \text{ an arbitrary lattice simplex}\right\} \,.\\
\end{align*}
\end{definition}

\noindent Clearly, for every $d \in \N$, we have
\[
\crke(d) \leq \crkh(d) \leq \crks(d) \,,
\]
and
\begin{align}
\crk_*(d) \leq \crk_*(d+1) \quad \textrm{for each choice of} \quad * \in \{\text{e}, \text{h}, \text{s}\} \,.\label{eqn:monotonicity-crks}
\end{align}
The maximal cyclicity rank is not an interesting parameter on the whole class of lattice simplices, nor on the class of hollow lattice simplices, because of the following observation:

\begin{proposition}
\label{prop:crk-general-and-hollow}
Let $\Delta \subseteq \R^d$ be a lattice simplex.
\begin{enumerate}[label={(\roman*)}]
 \item For every $d \in \N$, we have $\crks(d) = d$.
 \item $\crkh(1) = 0$ and $\crkh(d) = d$, for every $d \geq 2$.
\end{enumerate}
\end{proposition}

\begin{proof}
(i): If $\Delta = A S_d$ is a lattice simplex, then by~\eqref{eqn:quotient-smith-product} its quotient group is determined by the elementary divisors of~$A$.
In particular it consists of at most~$d$ factors, and hence, $\crk(\Delta) \leq d$.
An example attaining this upper bound is given by $\Delta = 2 S_d$, for any dimension $d \geq 1$.
In fact, we have $G_\Delta = (\Z_2)^d$ and hence $\crk(\Delta) = d$.

(ii): The simplex $\Delta = 2 S_d$ is hollow, whenever $d \geq 2$.
Thus, $\crkh(d) = d$, for $d \geq 2$.
Up to unimodular equivalence, the only hollow lattice $1$-simplex is the unit interval $[0,1]$.
This clearly has trivial quotient group, and hence $\crkh(1) = 0$.
\end{proof}

\noindent Our main interest thus lies in the case of \emph{empty} lattice simplices:

\begin{mproblem*}
Determine the constant $\crke(d)$ for a given dimension, or at least understand its asymptotic behavior in terms of~$d$.
\end{mproblem*}

Prior to this work, the parameter $\crke(d)$ is only known in dimensions at most four:
\begin{align}
\crke(1) = \crke(2) = 0 \quad \textrm{ and } \quad \crke(3) = \crke(4) = 1 \quad \textrm{ and } \quad \crke(5) \geq 2\,.\label{eqn:crke-dim-le-4}
\end{align}
This holds as every empty lattice $1$-simplex is unimodularly equivalent to $[0,1]$, and every empty lattice $2$-simplex is unimodularly equivalent to~$S_2$.
Empty lattice $3$-simplices have been characterized by White~\cite{white1964lattice} and come in a family that can be conveniently described by two parameters, allowing to read off that they are all cyclic.

\begin{theorem}[White~\cite{white1964lattice} (cf.~Seb\H{o}~\cite{sebo1999anintroduction})]
\label{thm:white}
Every empty lattice $3$-simplex is unimodularly equivalent to
\[
T(p,q) := \conv\left\{ \zero, (1,0,0)^\intercal,(0,1,0)^\intercal,(1,p,q)^\intercal\right\}\,,
\]
for some integers $1 \leq p < q$ with $\gcd(p,q)=1$.

In particular, there are two opposite edges of the simplex lying in parallel consecutive lattice planes, and thus these simplices belong to the class of Cayley polytopes.
\end{theorem}

The fact that also in dimension $d=4$, every empty lattice simplex is cyclic, has been proven by Barile et al.~\cite{barilebernardiborisovkantor2011onempty}, who also exhibited an example of a non-cyclic empty lattice $5$-simplex.
The four-dimensional result was instrumental for Iglesias-Vali\~{n}o \& Santos~\cite{iglesiassantos2021thecomplete} to achieve the complete classification of empty lattice $4$-simplices.

\section{Reductions and \texorpdfstring{$p$}{p}-power simplices}

In this section, we collect a few reductions for the main problem with the goal of identifying a very concrete class of lattice simplices that grasp the parameter $\crke(d)$ and are very convenient to deal with.
We start with the observation that given an empty lattice simplex corresponding to a sublattice $\Gamma \subseteq \Z^d$, every lattice simplex corresponding to an intermediate lattice is also empty.

\begin{lemma}
\label{lem:intermediate-empty}
Let $\Delta,\Delta' \subseteq \R^d$ be lattice simplices such that $\Gamma^\intercal_\Delta \subseteq \Gamma^\intercal_{\Delta'} \subseteq \Z^d$.
Then, if $\Delta$ is empty, then~$\Delta'$ is empty as well.
\end{lemma}

\begin{proof}
Let $\Delta = A S_d$ and $\Delta' = B S_d$.
By assumption on the lattices we have $A^\intercal \Z^d \subseteq B^\intercal \Z^d$.
For the polar lattices we thus have the reverse inclusion and get $B^{-1} \Z^d \subseteq A^{-1} \Z^d$.

Assume, that $\Delta'$ is non-empty. 
Then, for a non-vertex $w \in \Delta' \cap \Z^d$ of~$\Delta'$, the point $B^{-1}w\in S_d \cap B^{-1} \Z^d$ is not a vertex of $S_d$.
By the lattice inclusion from above, this implies that $B^{-1}w \in A^{-1}\Z^d$, and consequently the lattice point $AB^{-1} w \in A S_d \cap \Z^d$ is not a vertex of~$\Delta = A S_d$.
In other words, the simplex~$\Delta$ is non-empty, a contradiction.
\end{proof}

Next we reduce the problem to the study of empty lattice simplices whose quotient group is a power of a group of prime order.

\begin{lemma}
\label{lem:good-quotient}
Let $\Delta \subseteq \R^d$ be an empty lattice simplex with cyclicity rank $r = \crk(\Delta)$.
Then, there is a prime $p \in \N$ and an empty lattice simplex $\Delta' \subseteq \R^d$ such that $G_{\Delta'} \cong (\Z_p)^r$.
\end{lemma}

\begin{proof}
Let $\Delta = A S_d$ for some integer matrix $A$ with elementary divisors $m_1,\ldots,m_d$ that satisfy $m_d = \ldots = m_{r+1} = 1 < m_r \leq \ldots \leq m_1$.
Then, in view of~\eqref{eqn:quotient-smith-product} and \cref{prop:A-vs-A-transpose-quotients}, we have $\qGt{\Delta} \cong \qG{\Delta} \cong \Z_{m_r} \times \ldots \times \Z_{m_1}$.
Now, let $p$ be a prime dividing $m_r$, and thus dividing~$m_i$, for every $i \leq r$.
Since $\Z_p$ is isomorphic to the quotient $(\Z / m_i \Z)/(p \Z / m_i \Z)$ of~$\Z_{m_i}$, for $i \leq r$, we get that $H := (\Z_p)^r$ is isomorphic to a quotient of $\qGt{\Delta} = \Z^d / \,\Gamt{\Delta}$.
Therefore, there exists a sublattice $\Lambda_H \subseteq \Z^d$ with $\Gamt{\Delta} \subseteq \Lambda_H \subseteq \Z^d$ and $H \cong \Z^d / \Lambda_H$.
By virtue of \cref{lem:intermediate-empty}, every lattice simplex $\Delta' \subseteq \R^d$ such that $\Lambda_H = \Gamt{\Delta'}$ is an empty lattice simplex satisfying $\qG{\Delta'} \cong (\Z_p)^r$.
\end{proof}

Given a prime $p$, we call any (possibly non-empty) lattice simplex $\Delta$ with $G_\Delta \cong (\Z_p)^r$ a \emph{$p$-power simplex} for brevity.
The name is chosen to reflect that the normalized volume of such a simplex is given by $\Vol(\Delta) = \card{G_\Delta} = p^r$, and thus is a power of~$p$.

\subsection{The Hermite normal form of \texorpdfstring{$p$}{p}-power simplices}

We now aim to investigate $p$-power simplices more closely, and work towards identifying a representative in their unimodular equivalence class that allows to draw more information regarding their cyclicity rank.
To this end, we need a well-known result on manipulating integer matrices:

\begin{theorem}[{cf.~\cite[Thm.~1.2]{domichkannantrotter1987hermite}}]
\label{thm:HNF}
Let $A \in \Z^{d \times d}$ be an invertible integer matrix.
Then, there exists a unimodular matrix $U \in \Z^{d \times d}$ such that the matrix $H = UA = (h_{ij}) \in \Z^{d \times d}$ is upper triangular with positive diagonal entries, and its other entries satisfy
\[
0 \leq h_{ij} <  h_{jj} \quad \textrm{ for every } \quad 1 \leq i < j \leq d \,.
\]
\end{theorem}

The matrix~$H$ in \cref{thm:HNF} is uniquely determined and is called the \emph{Hermite normal form} of~$A$.
Hence, we can assume that we consider $p$-power simplices of the form $\Delta = \conv\{\zero,h_1,\ldots,h_d\}$, where the vertices $h_1,\ldots,h_d$ constitute the columns of a matrix $H = (h_1,\ldots,h_d)$ in Hermite normal form.
Furthermore, if a diagonal entry $h_{jj} = 1$, for some $1 \leq j \leq d$, then all other entries in its column need to be zero and all columns with index $i < j$ have only zero entries below the $j$-th row.
This implies, that with a possible change of rows and columns, we can assume
without loss of generality that the diagonal of~$H$ starts with a sequence of ones and continues with a sequence of entries strictly larger than one.
In particular, this means that the first, say~$k$, vertices of~$\Delta$ are the coordinate unit vectors in~$\Z^d$.
In summary, from here on out we assume that a $p$-power simplex~$\Delta$ is represented by a matrix in Hermite normal form with the shape
\begin{align}
H = \begin{pmatrix}
E_{d-k} & B \\
 0  & C
\end{pmatrix} \,,\label{eqn:HNF-shape-simplex}
\end{align}
for some $k \in \{0,1,\ldots,d\}$ and some matrices $B \in \Z^{(d-k) \times k}$ and $C \in \Z^{k \times k}$, where~$E_\ell$ denotes the $\ell \times \ell$ identity matrix.
An immediate but important observation is the following:

\begin{proposition}
\label{prop:triangle-shaped-quotient}
If $H \in \Z^{d \times d}$ has the form~\eqref{eqn:HNF-shape-simplex}, then
\[
\Z^d / H \Z^d \cong \Z^k / C \Z^k \,.
\]
\end{proposition}

\noindent We first investigate how the cyclicity rank of~$\Delta$ dictates the shape of the matrix~$C$ in~\eqref{eqn:HNF-shape-simplex}.
The following is the main auxiliary observation:

\begin{lemma}
\label{lem:p-power-cycrank-characterization}
Let $\Delta \subseteq \R^d$ be a $p$-power simplex in the form~\eqref{eqn:HNF-shape-simplex}.
Then, the following statements are equivalent:
\begin{enumerate}[label={(\roman*)}]
 \item $\crk(\Delta) = r$.
 \item The diagonal entries of~$H$ satisfy
 \[
 h_{1,1} = \ldots = h_{d-r,d-r} = 1 \quad \textrm{ and } \quad h_{d-r+1,d-r+1} = \ldots = h_{d,d} = p \,,
 \]
 and it holds
 \[
 \zangle{ e_{d-r+1} + H \Z^d,\ldots,e_d + H \Z^d } \cong (\Z_p)^r \,. 
 \]  
\end{enumerate}
\end{lemma}

\begin{proof}
Assume first, that $\crk(\Delta) = r$, that is, $G_\Delta \cong (\Z_p)^r$.
We show that $h_{j,j} \in \{1,p\}$, for every $1 \leq j \leq d$.
To this end, denote with $\bar e_i = e_i + H \Z^d$ the residue class of~$e_i$ in $\qG{\Delta} = \Z^d / H \Z^d$, and for $k \in \{2,\ldots,d\}$ consider the short exact sequence
\[
0 \ \to \ \zangle{\bar e_k} \cap \zangle{\bar e_1,\ldots,\bar e_{k-1}} \ \hookrightarrow \ \zangle{\bar e_k} \ \twoheadrightarrow \ \zangle{\bar e_1,\ldots,\bar e_k} / \zangle{\bar e_1,\ldots,\bar e_{k-1}} \ \to \ 0 \,.
\]
It holds that $\zangle{\bar e_1} \cong \Z_{h_{1,1}}$ and $\zangle{\bar e_1,\ldots,\bar e_k} / \zangle{\bar e_1,\ldots,\bar e_{k-1}} \cong \Z_{h_{k,k}}$, for $k \geq 2$, because~$H$ is upper triangular.
By the exactness of said sequence, the order of $\zangle{\bar e_1,\ldots,\bar e_k} / \zangle{\bar e_1,\ldots,\bar e_{k-1}}$ divides the order of $\zangle{\bar e_k}$.
The only cyclic subgroups of $(\Z_p)^r$ are $\{0\}$ and $\Z_p$, so that indeed we get $h_{j,j} \in \{1,p\}$, for every $1 \leq j \leq d$.

Now, in view of $h_{1,1}\cdot\ldots\cdot h_{d,d} = \det(H) = \card{G_\Delta} = p^r$, we obtain the claim on the diagonal elements in~(ii).
Moreover, we can view $\Z^d / H \Z^d = \zangle{ e_{d-r+1} + H \Z^d,\ldots,e_d + H \Z^d } \cong (\Z_p)^r$ as an $r$-dimensional vector space over the field with~$p$ elements, with basis given by $\left\{ e_{d-r+j} + H \Z^d : 1 \leq j \leq r \right\}$.

Conversely, assume that condition~(ii) holds.
Then, we have that $G_\Delta = \Z^d / H \Z^d = \zangle{ e_{d-r+1} + H \Z^d,\ldots,e_d + H \Z^d } \cong (\Z_p)^r$, and thus $\crk(\Delta) = r$ as claimed in~(i).
\end{proof}

The following example demonstrates the necessity for the elements $e_{\ell} + H \Z^{d}$, for $\ell = d-r+1,\ldots,d$, from the proof above to be independent in $\Z^{d} / H \Z^{d}$. 
It also shows that $\Z^{d} / H \Z^{d}$ may not be isomorphic to $ (\Z_{p})^{r}$ if they are not.

\begin{example}
Let
\[
H =
\left(\begin{array}{cc|cc}
1 & 0 & 1 & 0\\
0 & 1 & 1 & 0\\ \hline
0 & 0 & p & 1\\
0 & 0 & 0 & p
\end{array}\right) \,.
\]
Then, we get
\[
\Z^{4} / H \Z^{4} = \zangle{ e_3 + H \Z^4,e_4 + H \Z^4 }.
\]
But this matrix (respectively the $p$-power simplex spanned by its columns) does not satisfy the condition in \cref{lem:p-power-cycrank-characterization}~(ii), because by the last column of~$H$ we have $e_3 + p \, e_4 \in H\Z^4$, and thus $e_{3} \in -p \, e_{4} + H \Z^{4}$.
Hence, the quotient group $\Z^{4} / H \Z^{4}$ is cyclic with $\Z^{4} / H \Z^{4}= \zangle{ e_4 + H \Z^4 }\cong \Z_{p^{2}}$.
\end{example}

\begin{lemma}
\label{lem:p-power-diagonal-C-matrix}
Let $\Delta = H S_d \subseteq \R^d$ be a $p$-power simplex with $H \in \Z^{d\times d}$ in Hermite normal form in the shape~\eqref{eqn:HNF-shape-simplex}, that is, 
\[
H =
\begin{pmatrix}
E_{d-r} & B \\
      0 & C
\end{pmatrix} \,.
\] 
Then, $\crk(\Delta) = r$ if and only if $C = p \, E_{r}$.
\end{lemma}

\begin{proof}
Assume first, that $\crk(\Delta) = r$, that is, $\qG{\Delta} = \Z^d / H \Z^d \cong (\Z_p)^r$.
This implies that every element in~$\qG{\Delta}$ has order~$1$ or~$p$, so that~$\qG{\Delta}$ is a $\Z_p$-vector space of dimension~$r$.
In view of \cref{prop:triangle-shaped-quotient} we also have $\qG{\Delta} \cong \Z^r / C \Z^r$.
Now, for $1 \leq i \leq r$, let $\bar e_i = e_i + C \Z^r$ be the residue class of the standard unit vector~$e_i$ in~$\Z^r / C \Z^r$.
Clearly, $\bar e_1,\ldots,\bar e_r$ generate~$\Z^r / C \Z^r$ as an abelian group and thus also as a $\Z_p$-vector space, meaning that they form a basis of~$\Z^r / C \Z^r$.
Writing $C = (c_{i,j})_{1 \leq i,j \leq r}$, the $k$-th column of~$C$ gives the relation
\[
c_{1,k} \cdot \bar e_1 + \ldots + c_{r,k}  \cdot \bar e_r = 0
\]
in $\Z^r / C \Z^r$.
Hence, by the linear independence of the basis vectors~$\bar e_i$, all the entries of~$C$ equal $0$ modulo~$p$.
\cref{lem:p-power-cycrank-characterization} implies that the diagonal entries of~$C$ are all equal to~$p$, which together with the assumption that~$H$ is in Hermite normal form implies that $C = p \, E_r$, as desired.

%
%

The converse is direct from \cref{prop:triangle-shaped-quotient}, as $\qG{\Delta} = \Z^d / H \Z^d \cong \Z^r / (p E_r \Z^d) \cong (\Z_p)^r$.
This means, we have $\crk(\Delta) = r$ as desired.
\end{proof}

Note that \cref{lem:p-power-diagonal-C-matrix} holds for every $p$-power simplex, independently of it being empty.
A direct consequence is that every $p$-power simplex in~$\R^d$ with maximal cyclicity rank~$d$, is equivalent to a dilate of the standard simplex:

\begin{corollary}
Let $p$ be a prime and let $\Delta \subseteq \R^d$ be a $p$-power simplex with $\crk(\Delta) = d$, that is, $G_\Delta \cong (\Z_p)^d$.
Then, $\Delta$ is unimodularly equivalent to $p \, S_d$.
\end{corollary}

\begin{proof}
Assume that $\Delta = H S_d$ is given by $H$ in Hermite normal form.
Since $\crk(\Delta) = d$, \cref{lem:p-power-diagonal-C-matrix} implies that $H = p \, E_d$, and thus $\Delta = p \, S_d$ as claimed.
\end{proof}

We finish this section by collecting a few necessary conditions for a $p$-power simplex to be empty.

\begin{lemma}
\label{lem:necessary-conds-empty-p-power}
Let $\Delta = H S_d$ be a $p$-power simplex with  
\[
H =
\begin{pmatrix}
 E_{d-r} & B \\
 0       & p\, E_{r} 
\end{pmatrix} \,,
\]
for a suitable matrix $B \in \Z^{(d-r) \times r}$.
If $\Delta$ is an empty lattice simplex, then:
\begin{enumerate}[label={(\roman*)}]
 \item Every column of $B$ has at least two non-zero entries. 
 \item No two columns of~$B$ are integral multiples of one another.
 \item Every column of~$B$ containing exactly two non-zero entries is primitive.
 \item If $r = \crke(d)$, then every row of $B$ has a non-zero entry or $\crke(d) = \crke(d-1)$.
 \item For every submatrix $B' \in \Z^{(d-r) \times t}$ of $B$ with $t \leq \min(p,r)$, there exists an index $i \in [d-r]$ such that the sum of the entries of the $i$th row of~$B'$ is not divisible by $p$. 
\end{enumerate}
\end{lemma}

\begin{proof}
Write $k = d-r$ for brevity.
We let $H = (v_1,\ldots,v_d)$, so that $\Delta = \conv\{\zero,v_1,\ldots,v_d\}$.

(i): Since $\Delta$ is empty, the vectors $v_1,\ldots,v_d$ are primitive, so in particular, $v_j \neq p \cdot e_j$, for every $j \in \{k + 1,\ldots,d\}$.
Hence, every column of~$B$ is guaranteed to contain at least one non-zero entry.
Assume that the $j$th column of~$B$ has exactly one non-zero entry, say $1 \leq b_{ij} < p$, for some $i \in [k]$.
Then, 
\[
H \left(\frac{p - b_{ij}}{p} \, e_i + \frac{1}{p} \, e_{k+j} \right)
= \frac{p - b_{ij}}{p} \, e_i + \frac{1}{p} \left( b_{ij} \, e_i + p \, e_{k+j} \right)
= e_i + e_{k+j} \in\Z^{d} \setminus \{v_1,\ldots,v_d\} \,,
\]
which means that $\Delta$ is non-empty, because $\frac{p - b_{ij}}{p} \, e_i + \frac{1}{p} \, e_{k+j} \in S_d$.

(ii): Assume to the contrary that the $i$th column $b_i$ and the $j$th column~$b_j$ of~$B$ satisfy $b_i = \mu b_j$, for some $i \neq j$ and some $\mu \in \Z$.
Since $H$ is in Hermite normal form, every entry of~$B$ lies in $\{0,1,\ldots,p-1\}$ and thus $1 \leq \mu \leq p-1$.
Then,
\begin{align*}
H \left( \frac{1}{p} \, e_{k+i} + \frac{p-\mu}{p} \, e_{k+j} \right)
&= \frac{1}{p} \, ((b_i^\intercal,0,\ldots,0)^\intercal + p \, e_{k+i}) + \frac{p-\mu}{p} \, ((b_j^\intercal,0,\ldots,0)^\intercal + p \, e_{k+j}) \\
&= (b_j^\intercal,0,\ldots,0)^\intercal + e_{k+i} + (p-\mu) \, e_{k+j} \in \Z^{d} \setminus \{v_1,\ldots,v_d\} \,.
\end{align*}
Therefore, $\Delta$ is non-empty, because $\frac{1}{p} \, e_{k+i} + \frac{p-\mu}{p} \, e_{k+j} \in S_d$.

(iii): Assume that $b_i$ is a column of~$B$ with exactly two non-zero entries, which without loss of generality, we may assume to be the first two entries.
Thus, $b_i = (r,s,0,\ldots,0)^\intercal$, for some $1 \leq r,s < p$.
Since $\Delta$ is empty, the three-dimensional face $T = \conv\{\zero,e_1,e_2,r e_1 + s e_2 + p e_{d-r+i}\}$ of~$\Delta$ must be empty as well.
By White's characterization (\cref{thm:white}), a necessary condition for~$T$ to be empty is that two of its opposite edges lie in parallel consecutive lattice planes.
A simple calculation comparing each of the three pairs of opposite edges of~$T$ shows that this happens if and only if $r=1$ or $s=1$ or $r+s = p$.
Since~$p$ is prime, this implies that $\gcd(r,s)=1$ and hence $b_i$ is primitive.

(iv): Assume that $r = \crke(d)$ and that the $i$th row of~$B$ consists only of zero entries.
Deleting the $i$th row and $i$th column of~$H$, we then obtain the matrix $\widetilde{H} \in \Z^{(d-1) \times (d-1)}$, which, in view of \cref{lem:p-power-diagonal-C-matrix}, corresponds to an empty lattice $(d-1)$-simplex with the same cyclicity rank as~$\Delta$.
Hence with $\crke(\Delta) = \crke(d)$, we get $\crke(d) \leq \crke(d-1)$, and so by the monotonicity of $d \mapsto \crke(d)$ (see~\eqref{eqn:monotonicity-crks}), we obtain equality.

(v): Assume that there exists a submatrix $B' \in \Z^{k \times t}$ of~$B$ with $t\leq \min(p,r)$ such that the sum of the entries of every row of~$B'$ is divisible by~$p$.
Without loss of generality, we may assume that~$B'$ is given by the first~$t$ columns of~$B = (b_{ij})$.
Then,
\begin{align*}
H \cdot \sum_{j=1}^{t}\frac{1}{p} \, e_{k+j}
&= \sum_{i=1}^{k}\frac{1}{p}(b_{i1} + \ldots + b_{it}) \, e_i + \sum_{j=1}^{t}p\cdot \frac{1}{p} \, e_{k+j} \\
&= \sum_{i=1}^{k} \frac{b_{i1} + \ldots + b_{it}}{p} \, e_i + \sum_{j=1}^{t}e_{k+j} \in \Z^{d} \setminus \{v_1,\ldots,v_d\} \,,
\end{align*}
and thus~$\Delta$ is non-empty.
\end{proof}

\begin{remark}
\label{rem:necessary-conditions}
The conditions (ii) and (iii) in \cref{lem:necessary-conds-empty-p-power} are in the following sense best possible:
Condition~(ii) may be phrased as saying that any two columns of~$B$ are linearly independent.
There exist, however, empty $p$-power simplices whose matrix~$B$ contains three linearly dependent columns.
For a concrete example we refer to the first three columns of the defining matrix of the empty $3$-power simplex in \cref{prop:crk-3-d8}.

Likewise, condition~(iii) does not extend to columns of~$B$ with at least three non-zero entries.
Indeed, the $3$-power simplex $\conv\left\{\zero,e_1,e_2,e_3,(2,2,2,3)^\intercal\right\} \subseteq \R^4$ is empty.
\end{remark}

\section{Small Dimensions and the Case \texorpdfstring{$p=2$}{p=2}}
\label{sect:p=2}

In this section, we determine the constants $\crke(d)$ for dimensions $d \leq 7$.
This is based on understanding the monotonicity of the function $d \mapsto \crke(d)$, as well as the maximal cyclicity rank of an empty $2$-power simplex.

\begin{lemma}
\label{lem:facets-of-p-simplex}
Let $\Delta = H S_d = \conv\{\zero,v_1,\ldots,v_d\}$ be a $p$-power simplex with cyclicity rank $r = \crk(\Delta)$ and with
\[
H =
\begin{pmatrix}
E_{d-r} & B \\
      0 & p \, E_{r}
\end{pmatrix} \,,
\] 
for a suitable matrix $B \in \Z^{(d-r) \times r}$.
Furthermore, for every $i \in [d]$, we consider the facet $\Delta_i := \conv\{\zero,v_1,\ldots,v_{i-1},v_{i+1},\ldots,v_d\}$ of~$\Delta$, and its quotient group $G_{\Delta_i}$ defined with respect to the lattice $\Z^d \cap \lin(\Delta_i)$.
Then, we have
\[
G_{\Delta_j} \cong (\Z_p)^{r-1} \quad \textrm{ for every } \quad j \in \{d-r+1,\ldots,d\} \,.
\]
\end{lemma}

\begin{proof}
For $j \in \{d-r+1,\ldots,d\}$, we let $B(j) \in \Z^{(d-r) \times (r-1)}$ be the matrix obtained after deleting the column with index $j-(d-r)$ in~$B$, and let $H(j) \in \Z^{(d-1) \times (d-1)}$ be the matrix obtained after deleting the $j$th row and the $j$th column in~$H$.
Then, clearly
\[
H(j) =
\begin{pmatrix}
E_{d-r} & B(j) \\
      0 & p \, E_{r-1}
\end{pmatrix} \,,
\]
and $G_{\Delta_j} \cong \Z^{d-1} / \left(H(j) \, \Z^{d-1}\right)$, because the $(d-1)$-simplex $\bar \Delta_j := H(j) S_{d-1}$ arises from~$\Delta_j$ by deleting the $j$th coordinate.
Now, applying \cref{lem:p-power-diagonal-C-matrix} to~$\bar \Delta_j$ yields that $\crk(\Delta_j) = \crk(\bar \Delta_j) = r-1$, as desired.
\end{proof}

By this lemma we can now prove that the function $d \mapsto \crke(d)$ grows at most by one with each dimension.

\begin{theorem}
\label{thm:crk-empty-jumps}
For every $d \geq 2$ holds
\[
\crke(d) \leq \crke(d-1) + 1 \,.
\]
\end{theorem}

\begin{proof}
Let $\Delta$ be an empty lattice simplex with $\crk(\Delta) = \crke(d)$.
In view of \cref{lem:good-quotient} we may assume that $\Delta$ is a $p$-power simplex, for a suitable prime~$p$.
Moreover, we assume that $\Delta = H S_d$, where $H$ is a matrix satisfying the conclusion of \cref{lem:p-power-diagonal-C-matrix}.
Let $\Delta_d$ be the facet of~$\Delta$ that is contained in $\R^{d-1} \times \{0\}$.
Since~$\Delta$ is empty, we find that $\Delta_d$ is an empty lattice $(d-1)$-simplex considered with respect to the lattice $\Z^{d-1} \times \{0\}$.
Hence, \cref{lem:facets-of-p-simplex} implies that
\[
\crke(d) = \crk(\Delta) = \crk(\Delta_d) + 1 \leq \crke(d-1) + 1 \,,
\]
as claimed.
\end{proof}

As a direct corollary we obtain the maximal cyclicity rank of an empty lattice $5$-simplex, because $\crke(4) = 1$ and $\crke(5) \geq 2$ by the results of Barile et al.~\cite{barilebernardiborisovkantor2011onempty} (see~\eqref{eqn:crke-dim-le-4}).
Moreover, we get a first upper bound for arbitrary dimension that sets the constant $\crke(d)$ apart from its respective parameters~$\crkh(d)$ and $\crks(d)$ for hollow and arbitrary lattice simplices (see \cref{prop:crk-general-and-hollow}).

\begin{corollary}
\label{cor:crk5}\ 
\begin{enumerate}[label={(\roman*)}]
 \item $\crke(5) = 2$, and
 \item $\crke(d) \leq d-3$, for every $d \geq 4$.
\end{enumerate}
\end{corollary}

\subsection{The maximal cyclicity rank of empty \texorpdfstring{$2$}{2}-power simplices}

In this section, we explore the maximal cyclicity rank of $p$-power simplices for a fixed prime~$p$, focussing on the case $p=2$.
To this end, we define
\[
\crk_p(d) := \max\left\{\crk(\Delta) : \Delta \subseteq \R^d \text{ an empty } p\text{-power simplex}\right\} \,.
\]
Clearly, for every prime $p$ and every dimension $d$, we have $\crke(d) \geq \crk_p(d)$, and we may look for empty $p$-power simplices of high cyclicity rank for any specific fixed prime~$p$.
The following series of examples shows that
\begin{align}
\crk_p(k + \ell) \geq \ell \quad \textrm{ for every prime } p \textrm{ and every } k \geq 2 \textrm{ and } \ell \in [2^k - k - 1] \,.\label{eqn:exponential-lower-bound}
\end{align}
Note, that these examples generalize the Reeve simplices $\conv\{\zero,e_1,e_2,e_1 + e_2 + p \, e_3 \}$ in~$\R^3$ (see~\cite{reeve1957on}).

\begin{proposition}
\label{prop:crk-p-lower-bound}
Let $p$ be a prime, let $k \geq 2$ be an integer, and let $\ell \in [2^k - k - 1]$.
Define $B_\ell \in \Z^{k \times \ell}$ to be a matrix whose columns are pairwise different and constitute an $\ell$-element subset of $\{0,1\}^k \setminus \{\zero,e_1,\ldots,e_k\}$.
Then, writing
\[
H =
\begin{pmatrix}
E_k & B_\ell \\
      0 & p \, E_\ell
\end{pmatrix} \,,
\]
the $(k+\ell)$-simplex $\Delta = H S_{k+\ell}$ is empty and has $\crk(\Delta) = \ell$.
\end{proposition}

\begin{proof}
The fact that $\crk(\Delta) = \ell$ is immediate from the construction of~$H$ and \cref{lem:p-power-diagonal-C-matrix}.

So, let's prove that~$\Delta$ is empty.
Let $d = k + \ell$ and consider $\lambda = (\lambda_1,\ldots,\lambda_d) \in \R^d$ with $\lambda_i \geq0$, for every $i \in [d]$, and with $0 < |\lambda| := \lambda_1 + \ldots + \lambda_d \leq 1$.
We need to show that if $H \lambda \in \Z^d$, then $\lambda = e_j$, for some $j \in [d]$.

To this end, assume that $H \lambda = (a_1,\ldots,a_d) \in \Z^d$.
Because $B_\ell$ is a $0/1$-matrix, for every index $r \in [k]$ there exists a subset $I_r \subseteq \{k+1,\ldots,d\}$ such that
\[
a_r = \lambda_r + \sum_{i \in I_r} \lambda_i \,.
\]
Since $a_r \in \Z$ and $0 < |\lambda| \leq 1$, we thus have $a_r \in \{0,1\}$, for every $r \in [k]$.
Observe that, if $a_r = 1$, then $|\lambda| = 1$, and we have $\lambda_t = 0$, for every $t \in [d] \setminus (\{r\} \cup I_r)$.
This implies the following:
\begin{itemize}
 \item If there is $j \in [k]$ with $\lambda_j > 0$, then $a_j = 1$, and thus $\lambda_r = 0$, for every $r \in [k] \setminus\{j\}$.
 \item If there is $j \in \{k+1,\ldots,d\}$ with $\lambda_j > 0$, then $\lambda_r = 0$, for every $r \in [k]$.
 Indeed, by construction the $(j-k)$th column of~$B_\ell$ has at least two non-zero entries, say in rows~$r$ and~$s$, and thus $a_r = a_s = 1$.
\end{itemize}
If we combine these two observations, then we see that if there is an index $j \in [k]$ with $\lambda_j > 0$, then $\lambda = e_j$.

Now, we assume that $\lambda_1 = \ldots = \lambda_k = 0$ and that there are distinct indices $r,s \in \{k+1,\ldots,d\}$ such that $\lambda_r > 0$ and $\lambda_s > 0$.
Again, the $(r-k)$th and the $(s-k)$th column of~$B_\ell$ contain at least two non-zero entries.
Each such non-zero entry, say in row~$q$, forces $\lambda_i = 0$, for all the indices~$i$, so that the entry of~$B_\ell$ in row~$q$ and column~$i-k$ equals~$0$.
This is however only possible if the $(r-k)$th and the $(s-k)$th column of~$B_\ell$ coincide, contradicting the definition of~$B_\ell$.
As a result, the vector $\lambda$ has exactly one non-zero entry, say $\lambda_j > 0$ for some $j \in \{k+1,\ldots,d\}$, the corresponding column of~$B_\ell$ has a non-zero entry in the $r$th row, and we have $a_r = 1 = \lambda_j$.
Therefore, $\lambda = e_j$ as desired.
\end{proof}

In the case $p=2$ the necessary conditions in \cref{lem:necessary-conds-empty-p-power} imply that the simplices in \cref{prop:crk-p-lower-bound} have the maximal cyclicity rank among all $2$-power simplices.

\begin{theorem}
\label{thm:crk-2-exact}
For every $d \geq 3$, we have
\[
\crk_2(d) = d - \lfloor \log_2(d) \rfloor - 1 \,.
\]
\end{theorem}

\begin{proof}
Let $\Delta = H S_d$ be an empty $2$-power simplex.
Letting $\crk(\Delta) = \ell$, we can write $d = k + \ell$, for some $k \geq 3$, by the upper bound in \cref{cor:crk5}~(ii).
Using \cref{lem:necessary-conds-empty-p-power}~(i) and~(ii), we see that the matrix~$B \in \{0,1\}^{k \times \ell}$ in the representation
\[
H =
\begin{pmatrix}
E_k & B \\
      0 & 2 \, E_\ell
\end{pmatrix} \,,
\]
has pairwise distinct columns, all of which have at least two non-zero entries.
Hence, $\crk(\Delta) = \ell \leq 2^k - k - 1$.

This holds for every partition of $d = k + \ell$, in particular for the one minimizing~$k$.
More precisely, if $k$ is such that $2^{k-1} \leq d < 2^k$, then $\ell = d - k = d - (\lfloor \log_2(d) \rfloor + 1)$.
The construction leading to~\eqref{eqn:exponential-lower-bound} shows that this upper bound on $\crk_2(d)$ is attained.
\end{proof}

Together with \cref{cor:crk5}, we thus obtained the exact value of $\crke(d)$, for all $d \leq 7$.

\begin{corollary}
\label{cor:crke-p2-dim-le-7}
$\crke(6) = 3$ and $\crke(7) = 4$.
\end{corollary}

\section{Large dimensions and the case \texorpdfstring{$p=3$}{p=3}}
\label{sect:p=3}

Starting from dimension $d = 8$, empty $2$-power simplices alone do not describe the function $\crke(d)$ anymore, so that we need to pass to larger primes~$p$.
In this last section, we explore this for $d \in \{8,9\}$, and asymptotically for large~$d$ and fixed~$p$.

By extending the argument in the proof of \cref{thm:crk-2-exact}, we can give an upper bound on the cyclicity rank of empty $p$-power simplices, for arbitrary primes~$p$.

\begin{proposition}
\label{prop:crk-p-upper-bound}
For every prime $p$, we have
\[
\crk_p(d) \leq d - \lfloor \log_p(d) \rfloor - 1 \,.
\]
\end{proposition}

\begin{proof}
Let $\Delta = H S_d$ be an empty $p$-power simplex.
Letting $\crk(\Delta) = \ell$, we can write $d = k + \ell$, for some $k \geq 3$, by the upper bound in \cref{cor:crk5}~(ii).
Using \cref{lem:necessary-conds-empty-p-power}~(i) and~(ii), we see that the matrix~$B \in \{0,1,\ldots,p-1\}^{k \times \ell}$ in the representation
\[
H =
\begin{pmatrix}
E_k & B \\
      0 & p \, E_\ell
\end{pmatrix} \,,
\]
has pairwise distinct columns, all of which have at least two non-zero entries.
Hence, $\crk(\Delta) = \ell \leq p^k - (p-1)k - 1$.
Parts~(ii) and~(iii) of \cref{lem:necessary-conds-empty-p-power} actually allow for a more precise upper bound, which however leads to the same asymptotic result.

Observe that the bound holds for every partition of $d = k + \ell$, in particular for the one minimizing~$k$.
This corresponds to choosing~$k$ such that
\begin{align}
p^{k-1} - (p-2)(k-1) \leq d < p^k - (p-2)k \,.
\end{align}
In particular, $d < p^k$, so that $k \geq \lfloor \log_p(d) \rfloor + 1$, and hence $\crk(\Delta) = \ell = d - k \leq d - \lfloor \log_p(d) \rfloor - 1$ as claimed.
\end{proof}

Together with the lower bound on $\crke(d)$ implied by \cref{thm:crk-2-exact}, the previous upper bound suggests the following question:

\begin{question}
\label{qu:crke-asymptotics}
Is the asymptotic behavior of the maximal cyclicity rank of empty $d$-simplices given by
\[
\crke(d) = d - \Theta(\log(d)) \,?
\]
\end{question}

One may approach this problem as follows:
Let $p(d)$ be the smallest prime number with $\crke(d) = \crk_{p(d)}(d)$, which exists in view of \cref{lem:good-quotient}.
With this notation, \cref{prop:crk-p-upper-bound} gives
\[
\crke(d) \leq d - \Omega\left(\frac{\log(d)}{\log(p(d))}\right) \,,
\]
so that the question boils down to asking whether $p(d)$ is bounded by a constant, independently of~$d$.
While this might be unlikely, any bound of order $p(d) \in o(d)$, would give an asymptotic improvement on the linear upper bound in \cref{cor:crk5}~(ii).

On a different note, observe that any explicit computable upper bound on~$p(d)$ would leave only finitely many empty $d$-simplices to investigate (see, for instance~\cref{lem:p-power-diagonal-C-matrix}).
In particular, the number~$\crke(d)$ would be algorithmically computable in fixed dimension~$d$.

%

\subsection{Some results for empty \texorpdfstring{$3$}{3}-power simplices}

In dimension $d=8$, \cref{thm:crk-2-exact} gives $\crke(8) \geq \crk_2(8) = 4$, while \cref{cor:crk5} shows that $\crke(8) \leq 5$.
It turns out that the upper bound gives the correct value, which can be shown by exhibiting a concrete empty $p$-power simplex~$\Delta \subseteq \R^8$ with $\crk(\Delta) = 5$, for some prime $p \geq 3$:

\begin{proposition}
\label{prop:crk-3-d8}
We have $\crke(8) = \crk_3(8) = 5$.

Moreover, up to unimodular equivalence, the simplex $\Delta_8 = H S_8$ with
\[
H =
\begin{pmatrix}
E_3 & B \\
      0 & 3 \, E_5
\end{pmatrix}
\quad \textrm { and } \quad
B = \begin{pmatrix}
1 & 0 & 1 & 1 & 2 \\
0 & 1 & 1 & 2 & 1 \\
1 & 1 & 2 & 2 & 2
\end{pmatrix}
\]
is the only empty $3$-power simplex with cyclicity rank~$5$.
\end{proposition}

\begin{proof}
By construction we clearly have $\crk(\Delta_8) = 5$.
The fact that $\Delta_8$ is empty can be checked, for instance, by with the help of a computer using \texttt{sagemath}~\cite{sagemath}.%
\footnote{A pen and paper proof that $\Delta_8$ is empty can be found in the master thesis of the first author~\cite[Ex.~3.3.1]{abend2023masterthesis}.}

In order to check the claimed uniqueness of~$\Delta_8$ we first observe that the conditions in \cref{lem:necessary-conds-empty-p-power} imply that it suffices to consider matrices~$B$ that are composed of a choice of~$5$ of the columns of the following matrix:
\begin{align}
\setcounter{MaxMatrixCols}{20}
\begin{pmatrix}
1 & 1 & 0 & 1 & 2 & 1 & 2 & 0 & 0 & 1 & 1 & 1 & 2 & 1 & 2 & 2 & 2 \\
1 & 0 & 1 & 2 & 1 & 0 & 0 & 1 & 2 & 1 & 1 & 2 & 1 & 2 & 1 & 2 & 2 \\
0 & 1 & 1 & 0 & 0 & 2 & 1 & 2 & 1 & 1 & 2 & 1 & 1 & 2 & 2 & 1 & 2
\end{pmatrix} \,.\label{eqn:matrix-of-B-columns}
\end{align}
This amounts to checking emptiness of a total of $\binom{17}{5} = 6188$ lattice $8$-simplices.
Using \texttt{sagemath} again this can be done in about $30$ minutes of computer time, and results in a list of $18$ possibilities for ``empty'' matrices~$B$.
Checking unimodular equivalence of two lattice polytopes can be done via the concept of a normal form as proposed by Kreuzer \& Skarke~\cite{kreuzerskarke1998classification} (see also~\cite{griniskasprzyk2013normal}), whose algorithm has been implemented in \texttt{sagemath}.
It turns out that the $18$ empty lattice simplices above are pairwise unimodularly equivalent.
The code and results of the aforementioned computer calculations can be found at \url{https://github.com/mschymura/cyclicity-rank-computations}.
\end{proof}

\begin{remark}
The matrix~$B$ in the construction of \cref{prop:crk-3-d8} is inspired by Harborth's example showing that $f(3,3) \geq 19$, where $f(n,d)$ denotes the smallest integer~$m$ such that any choice of~$m$ elements from $(\Z_n)^d$ (repetitions are allowed) contains~$n$ elements that sum to zero in~$(\Z_n)^d$.
In fact, Harborth's example consists of the $9$ vectors in~$(\Z_3)^3$ that correspond to projecting the~$9$ vertices of~$\Delta_8$ to the first three coordinates (see~\cite{elsholtz2004lowerbounds,harborth1973einextremal} for more information on the multidimensional zero-sum problem around the function $f(n,d)$).

We do not know whether this connection is a coincidence, or whether there is a deeper connection between the problems around $f(n,d)$ and $\crk_p(d)$.
\end{remark}

With such a computational approach, we found that lifting the simplex in \cref{prop:crk-3-d8} to an empty $3$-power simplex gives the maximal cyclicity rank in dimension~$9$.
This yields $\crke(9) \in \{5,6\}$, but we do not know whether there is a prime $p > 3$ such that $\crk_p(9) = 6$.

\begin{proposition}
\label{prop:p=3-dim8+9}
$\crk_2(9) = \crk_3(9) = 5$.%
\footnote{A proof without computer assistance again can be found in~\cite[Thm.~3.3.5]{abend2023masterthesis}.}
\end{proposition}

\begin{proof}
The fact that $\crk_2(9) = 5$ holds by \cref{thm:crk-2-exact}.
The simplex $\Delta_9 = H S_9$ with
\[
H =
\begin{pmatrix}
E_4 & B \\
      0 & 3 \, E_5
\end{pmatrix}
\quad \textrm { and } \quad
B = \begin{pmatrix}
0 & 0 & 0 & 0 & 0 \\
1 & 0 & 1 & 1 & 2 \\
0 & 1 & 1 & 2 & 1 \\
1 & 1 & 2 & 2 & 2
\end{pmatrix}
\]
is empty by \cref{prop:crk-3-d8} and has $\crk(\Delta_9) = 5$.
In order to show that $\crk_3(9) \leq 5$, we consider the possibly empty $3$-power simplices~$\Delta$ of $\crk(\Delta) = 6$, each of which correspond to a choice of~$6$ columns from the matrix in~\eqref{eqn:matrix-of-B-columns}.
This leaves us with $\binom{17}{6} = 12376$ examples, and another computation with \texttt{sagemath} shows that none of these simplices is empty.

Unlike for dimension~$8$, it is not feasible to compute all empty $3$-power $9$-simplices with cyclicity rank~$5$, without significant further reductions.
Indeed, one would need to check all choices of~$5$ columns of a list of~$59$ vectors $b \in \{0,1,2\}^4$ left by \cref{lem:necessary-conds-empty-p-power}.
These are more than $5$ million simplices, with an expected one month of computation time on a standard computer.
\end{proof}

As a final result, we describe how to lift the example in \cref{prop:crk-3-d8} to empty $3$-power simplices in higher dimensions and which have a cyclicity rank larger than any $2$-power simplex of the same dimension.

\begin{proposition}
\label{prop:lift-empty-3-powers}
Let $k,\ell \in \N$ be such that $\crk_3(k+\ell) \geq \ell$.
Then, $\crk_3(k+1+m) \geq m$, where $m = 2\ell + k$.
\end{proposition}

\begin{proof}
Let $\Delta = H S_{k+\ell}$ be an empty $3$-power simplex with $\crk(\Delta) = \ell$ and with
\[
H =
\begin{pmatrix}
E_k & B \\
      0 & 3 \, E_\ell
\end{pmatrix} \,,
\]
for a suitable matrix $B \in \{0,1,2\}^{k \times \ell}$.
Let $\tilde \Delta = \tilde H S_{k+1+m}$ be defined by the matrix

\newcommand\coolunder[2]{\mathrlap{\smash{\underbrace{\phantom{%
    \begin{matrix} #2 \end{matrix}}}_{\mbox{$#1$}}}}#2}

\[
\tilde H =
\begin{pmatrix}
E_{k+1} & \tilde B \\
      0 & 3 \, E_m
\end{pmatrix}
= 
\begin{pNiceMatrix}[last-row]
E_k & 0 & E_k      &         B   & B \\
  0 & 1 & \one_k   & \one_\ell   & 2 \cdot \one_\ell \\
  0 & 0 & 3 \, E_k & 0           & 0 \\
  0 & 0 & 0        & 3 \, E_\ell & 0 \\
  0 & 0 & 0        & 0           & 3 \, E_\ell \\
    &   & \textrm{\vphantom{\Huge M}\footnotesize Block E} & \textrm{\footnotesize Block B1} & \textrm{\footnotesize Block B2}
\rule[-8pt]{0pt}{8pt}
\CodeAfter
\UnderBrace[yshift=2pt]{1-3}{5-3}{}
\UnderBrace[yshift=2pt]{1-4}{5-4}{}
\UnderBrace[yshift=2pt]{1-5}{5-5}{}
\end{pNiceMatrix}
\,,
\]
where $\one_r$ denotes the all-one row-vector with~$r$ entries.
The names of the three blocks of columns are intended for a more convenient description of the arguments below.
By definition, we have $\crk(\tilde \Delta) = m$, so that we are left to show that~$\tilde \Delta$ is empty.

To this end, assume for contradiction that $\tilde \Delta$ contains a lattice point besides its vertices.
Letting $\tilde h_1,\ldots,\tilde h_{k+1+m}$ be the columns of~$\tilde H$, such a lattice point $z \in \tilde\Delta$ has a representation of the form
\[
z = \sum_{j=1}^{k+1} \lambda_j e_j + \sum_{i=1}^m \mu_i \tilde h_{k+1+i} \,,
\]
where $0 \leq \lambda_j < 1$, for $j \in [k+1]$, and $\mu_i = \frac{n_i}{3}$, for suitable $n_i \in \{0,1,2\}$ and $i \in [m]$, and with $0 < \lambda_1 + \ldots + \lambda_{k+1} + \mu_1 + \ldots + \mu_m \leq 1$.
In particular, there are at most three indices $i \in [m]$ such that $\mu_i \neq 0$.

\noindent \emph{Case 1:} $\mu_i = 0$, for every $i \in [m]$.

This means that the first $k+1$ coordinates of $z = \sum_{j=1}^{k+1} \lambda_j e_j$ correspond to a non-vertex lattice point in $S_{k+1}$, which is a contradiction.

\noindent \emph{Case 2:} $\mu_r \neq 0$, for exactly one index $r \in [m]$.

If $\mu_r$ is the coefficient of a column in Block B1 or B2, then omitting the $(k+1)$st coordinate shows that~$z$ projects onto a non-vertex lattice point in $\Delta = H S_{k+\ell}$; a contradiction.

If $\mu_r$ corresponds to Block E, then $\lambda_{k+1} = 1 - \mu_r$, because $z_{k+1} \in \Z$.
Thus, $z = (1-\mu_r) e_{k+1} + \mu_r \tilde h_{k+1+r}$ as a lattice point in $\R^{k+1+k}$ lies in the relative interior of an edge of the simplex defined by the upper left square submatrix of size $k+1+k$ of~$\tilde H$.
This is a contradiction as all edges of that simplex are primitive.

\noindent \emph{Case 3:} $\mu_r,\mu_s \neq 0$, for exactly two indices $r,s \in [m]$.

Observe that $\mu_r + \mu_s \in \{\frac23,1\}$.
If $\mu_r + \mu_s = 1$, then because of the $(k+1)$st coordinate of~$z = \mu_r \tilde h_{k+1+r} + \mu_s \tilde h_{k+1+s}$, which must be integral, the indices $r,s$ correspond to columns that are either both in Block B2 or both not in Block B2.
Since $B$ has the property that each of its columns has at least two non-zero entries (see \cref{lem:necessary-conds-empty-p-power}), these entries are in $\{1,2\}$, and because of the unit matrix~$E_k$ in Block E of~$\tilde H$, this means that the indices $r,s$ correspond to columns that are either both in Block B1 or both in Block B2.
This, however, contradicts the assumed emptiness of the simplex~$\Delta$.

So, let $\mu_r + \mu_s = \frac23$, that is, $\mu_r = \mu_s = \frac13$.
It cannot be that both~$r$ and~$s$ correspond to columns in Block B2, since then $\lambda_{k+1} = \frac23$ to get~$z_{k+1}$ integral, which contradicts $\lambda_{k+1} + \mu_r + \mu_s \leq 1$.
If, say, $r$ corresponds to Block B2, and~$s$ does not, then $\lambda_{k+1} = 0$.
In the case that~$s$ corresponds to Block E, neglecting rows $k+1,\ldots,k+1+k$ of~$\tilde H$ shows that~$z$ points to a non-vertex lattice point in~$\Delta$; a contradiction.
In the case that~$s$ corresponds to Block B1, two things can happen.
First, $r$ and~$s$ select different columns in the matrix~$B$ in Blocks B1 and B2.
Here, again~$z$ points to a non-vertex lattice point in~$\Delta$.
Second, $r$ and~$s$ select the same column of~$B$.
Here, the part $\mu_r \tilde h_{k+1+r} + \mu_s \tilde h_{k+1+s}$ of~$z$ has at least two non-zero entries in $\{\frac23,\frac43\}$ in the first~$k$ rows, which cannot both be complemented to an integer by the remaining part $\lambda_1 e_1 + \ldots + \lambda_k e_k$, because $\lambda_1 + \ldots + \lambda_k \leq \frac13$.
If neither~$r$ nor~$s$ correspond to Block B2, then $\lambda_{k+1} = \frac13$ and hence $z = \frac13 e_{k+1} + \frac13 \tilde h_{k+1+r} + \frac13 \tilde h_{k+1+s}$.
Then, there is no constraint from the integrality of~$z$ in the rows $k+1,\ldots,k+1+k$, and~$z$ again points to a non-vertex lattice point in~$\Delta$.

\noindent \emph{Case 4:} $\mu_r,\mu_s,\mu_t \neq 0$, for exactly three indices $r,s,t \in [m]$.

This case forces $\mu_r = \mu_s = \mu_t = \frac13$ and thus $z = \frac13 \tilde h_{k+1+r} + \frac13 \tilde h_{k+1+s} + \frac13 \tilde h_{k+1+t}$.
Looking at the $(k+1)$st coordinate of~$z$ shows that either $r,s,t$ all correspond to a column in Block~B2 or none of them does.
If all of them belong to Block~B2, then only considering the first~$k$ and last~$\ell$ rows of~$\tilde H$ yields a non-vertex lattice point in~$\Delta$.
Similarly, if all three indices $r,s,t$ belong to Block E or Block~B1, then there is no constraint in rows $k+1,\ldots,k+1+k$ and $k+1+k+\ell+1,\ldots,k+1+k+2\ell$, and hence the remaining $k+\ell$ rows of~$\tilde H$ yield a non-vertex lattice point in~$\Delta$ once again.
\end{proof}

Iterating the construction in \cref{prop:lift-empty-3-powers} yields a sequence $\left(d(k)\right)_{k \geq 3}$ of dimensions, such that $\crk_3(d(k)) \geq d(k) - k$, and which satisfies the recurrence
\[
d(3) = 8 \quad\textrm{ and }\quad d(k) = 2 d(k-1) + 1 \,,\quad \textrm{ for } \quad k \geq 4 \,,
\]
with the initial value being due to \cref{prop:crk-3-d8}.
Solving this recurrence shows that
\[
\crk_3(2^k + 2^{k-3} - 1) \geq 2^k + 2^{k-3} - k - 1 = \crk_2(2^k + 2^{k-3} - 1) + 1 \,, \textrm{ for } k \geq 3 \,.
\]
In particular, there are infinitely many dimensions~$d$ for which $\crk_2(d) < \crk_3(d)$.

In general, one might expect such a monotonic behavior of $\crk_p(d)$ with respect to any dimension and primes~$p$.
\begin{question}
\label{qu:crk-p-monotonic}
Is the maximal cyclicity rank of $p$-power simplices non-decreasing with~$p$?
More precisely, do we have $\crk_p(d) \leq \crk_q(d)$, for every $d \in \N$ and every two primes $p \leq q$?
\end{question}

Note however, that things are less clear than at first sight:
In an empty $p$-power simplex defined by $H = \begin{pmatrix}
E_{d-r} & B \\
      0 & p \, E_r
\end{pmatrix}$, we cannot just replace $p \, E_r$ with $q \, E_r$ and obtain an empty $q$-power simplex.
For a concrete example, take the empty $3$-power simplex $\conv\left\{\zero,e_1,e_2,e_3,(2,2,2,3)^\intercal\right\}$ from \cref{rem:necessary-conditions}.
It turns out that the $5$-power simplex $\conv\left\{\zero,e_1,e_2,e_3,(2,2,2,5)^\intercal\right\}$ is not empty, since it contains $(1,1,1,2)^\intercal$ as a lattice point additionally to its vertices.


\bibliographystyle{amsplain}
\bibliography{mybib}

\end{document}